\documentclass[11pt]{amsart}
\usepackage[utf8]{inputenc}
\usepackage{amsmath}
\usepackage{amsfonts}
\usepackage{amssymb}
\usepackage{amsthm}
\usepackage{csquotes}
\usepackage[english]{babel}
\usepackage[shortlabels]{enumitem}
\usepackage{lmodern}
\usepackage{cite}
\PassOptionsToPackage{hyphens}{url}\usepackage{hyperref}

\usepackage{color}

\usepackage{tikz}
\usetikzlibrary{decorations.markings}
\usetikzlibrary{patterns}

\newtheorem{theorem}{Theorem}[section]

\newtheorem{corollary}[theorem]{Corollary}
\newtheorem{lemma}[theorem]{Lemma}
\newtheorem{proposition}[theorem]{Proposition}

\theoremstyle{definition}

\theoremstyle{remark} \theoremstyle{remark}
\newtheorem{remark}[theorem]{Remark}

\newtheorem*{theorem*}{Theorem}

\DeclareMathOperator{\id}{id}
\DeclareMathOperator{\End}{End}

\newcommand{\CC}{{\mathbb C}}
\newcommand{\PP}{{\mathbb P}}
\newcommand{\SO}{{\mathrm{SO}}}
\newcommand{\Sp}{{\mathrm{Sp}}}

\newcommand{\Un}{{\mathrm{U}}}
\newcommand{\SU}{{\mathrm{SU}}}

\newcommand{\HH}{{\mathbb{H}}}
\newcommand{\RR}{{\mathbb{R}}}
\newcommand{\ZZ}{{\mathbb{Z}}}

\numberwithin{equation}{section}

\begin{document}

\title[AQS on CROSSes]{Almost quaternionic structures on \\ compact rank one symmetric spaces}
 
 \author[Goertsches]{Oliver Goertsches}  
  \address{Philipps-Universit\"at Marburg, Fachbereich Mathematik und Informatik, Hans-Meerwein-Straße 6, 35043 Marburg, Germany}
 \email{goertsch@mathematik.uni-marburg.de}
 \author[Konstantis]{Panagiotis Konstantis}
  \email{pako@mathematik.uni-marburg.de}
 \author[Loiudice]{Eugenia Loiudice}
  \email{loiudice@mathematik.uni-marburg.de}
 \author[Solomadin]{Grigory Solomadin}
	\email{grigory.solomadin@gmail.com}

\begin{abstract}
	We prove that the only CROSSes that admit a (not necessarily homogeneous) almost quaternionic structure are $\HH P^n$ and $\CC \PP^2$. 
\end{abstract}

\maketitle

\section{Introduction}

This note is motivated by the question posed by Moroianu \cite{overflowmoroianu} which topological obstructions to the existence of an almost quaternionic structure exist on a compact manifold, beyond the fact that in real dimensions $8n$ such structures force the manifold to be spin, see Marchiafava--Romani \cite{marchiafava-romani}, and the conditions in dimension $8$ found by Čadek--Vanžura \cite{CadekVanzura}. In particular, it was asked in \cite{overflowmoroianu} whether a complex projective space other than $\CC \PP^2$ could admit such a structure. 

In this note we settle this latter question by proving:

\begin{theorem}\label{thm:main}
	Let $M$ be a compact rank one symmetric space. If $M$ admits an almost quaternionic structure, then $M$ is either $\HH P^n$ or $\CC \PP^2$.
\end{theorem}

We emphasize that the almost quaternionic structures we consider are not assumed to be invariant under any Lie group action. Concerning homogeneous almost quaternionic structures, the most comprehensive result in the literature was obtained in \cite{MoroianuPilcaSemmelmann}, where Moroianu--Pilca--Semmel\-mann classify the compact simply-connected equal rank homogeneous spaces admitting such a structure. Note also that it is well-known that compact rank one symmetric spaces other than $\HH P^n$ and $\CC \PP^2$ cannot admit a positive quaternion-Kähler structure, as such a structure on a compact manifold $M$, in case $M$ is neither $\HH P^n$ nor the complex Grassmannian ${\mathrm{Gr}}_2(\CC^{n+2})$, forces the second homotopy group $\pi_2(M)$ to be a nontrivial torsion group, see LeBrun--Salamon \cite[Theorem 0.2]{LeBrunSalamon}.

It was already noted in \cite{overflowmoroianu} that by work of Sutherland \cite[Theorem 1.4]{Sutherland} spheres $S^{4n}$ ($n\geq 2$) cannot admit any almost quaternionic structure, as no $S^2$-bundle over them -- and hence no potential twistor space -- admits an almost complex structure. However, the proof we give below does not make use of this fact. Also, the spin condition in dimension $8n$ \cite{marchiafava-romani} immediately rules out the complex projective spaces $\CC \PP^{4n}$. 

Our first idea for the proof of Theorem \ref{thm:main} is to recursively compute the Chern classes of the horizontal bundle $H$ over the twistor space making use of the isomorphism $V\otimes \overline{H}\cong H$ of complex vector bundles (see Proposition \ref{prop:VHquerH}), where $V$ is the vertical bundle over the twistor space, in order to obtain a contradiction to the equality between the top Chern and the Euler class. This strategy yields immediate success for the spheres, see Section \ref{sec:spheres}, and the Cayley plane, see Section \ref{sec:cayley}. It also gives short proofs of a divisibility result for the first Chern class of the twistor space, see Proposition \ref{prop:c1}, proven by Kobayashi \cite{kobayashi} for complex contact manifolds, as well as for the above-mentioned fact that $8n$-dimensional almost quaternionic manifolds are spin, see Corollary \ref{Cor: AQ in dim 8n are spin}.

This recursive approach is also applicable for the remaining complex projective spaces $\CC \PP^{4n+2}$, but with this method we were only able to obtain a partial result, which is explained in Section \ref{rem:remark} below.  Consequently, we follow in Section \ref{sec:cpn} a different approach for these spaces: using complex K-theory we are able to derive a contradiction through a certain tensor product decomposition of the complexified tangent bundle (see Proposition \ref{P: TM_C = E ox F}).
\\

\noindent {\bf Acknowledgements.}  We gratefully acknowledge funding of the DFG (German Research Foundation): Project numbers 561158824 (Walter Benjamin Fellowship of G.S.) and 452427095 (O.G and P.K.).

\section{Preliminaries}

\subsection{Almost quaternionic structures} \label{subsec:aqs}

In this section we recall the basic notions of almost quaternionic geometry. This material is well-known, see e.g.\ \cite{AMP}. Let $M$ be a smooth manifold of dimension $4n$. An \emph{almost quaternionic structure} on $M$ is a subbundle $Q\subset \End(TM)$ of rank $3$ with the property that it is locally spanned by an almost hypercomplex structure, namely by (local) sections $I,J,K$ of $Q$ which are almost complex structures and satisfy $IJ=-JI = K$. Any such local base of $Q$ is called a \emph{(local) quaternionic base}. As pointwise any two such bases differ by a matrix in $\SO(3)$, we can introduce on the bundle $Q$ a metric defined by the property that every local quaternionic base is orthonormal.

The \emph{twistor space} of $M$ is the unit sphere bundle $\pi:Z\to M$ of $Q$. It is an $S^2$-bundle over $M$. Denote by $V\to Z$ the vertical bundle of the twistor fibration, i.e., $V_I= \ker d\pi_I$, for $I\in Z$. Then naturally
\begin{equation}\label{eq:fiberV}
V_I = \{J\in Q_{\pi(I)}\mid IJ=-JI\}.
\end{equation}
We have the following isomorphism (which in fact holds true for any sphere bundle in a vector bundle):
\begin{lemma}\label{lem:QV}
	We have $\pi^*Q \cong V\oplus \underline{\RR}$.	
\end{lemma}
\begin{proof}
	The map $V\oplus \underline{\RR}\to \pi^*Q$ given by $V_I\oplus {\RR} \to Q_{\pi(I)};\, (J,t)\mapsto J+tI$ is an isomorphism.
\end{proof}
In this theory, one often chooses a horizontal subbundle $H\to Z$ associated to a connection on $M$ leaving $Q$ invariant; see, e.g., \cite{AMP}. For us, this will not be relevant: it is sufficient to choose any subbundle $H\subset TZ$ complementary to $V$.

Both bundles $V$ and $H$ are naturally equipped with an almost complex structure: every element  $I\in Z$ is an almost complex structure on $T_{\pi(I)}M$ and hence induces an almost complex structure $J_I^H$ on $H_I$ via the isomorphism $d\pi_I:H_I\to T_{\pi(I)}M$. On $V_I$ it defines the complex structure $J\mapsto I J$ (using \eqref{eq:fiberV}). Hence, $Z$ becomes an almost complex manifold via the isomorphism $TZ = V\oplus H$. For further reference, we note
\begin{lemma} \label{lem:identitiesH}
	We have
	\[
	TZ = V\oplus H
	\]
	as complex vector bundles. Moreover, as real vector bundles
	\[
	H\cong \pi^*TM.
	\]
\end{lemma}

\subsection{Pontryagin and Chern classes}

The following two lemmas will be used in the proof of Theorem \ref{thm:main}, in order to compute the Chern classes of the horizontal bundle $H\to Z$ over the twistor space of an almost quaternionic manifold $M$. They are valid for complex vector bundles over any space $X$. Throughout this note, we consider singular cohomology with integer coefficients unless explicitly specified otherwise.

\begin{lemma} \label{lem:pontryaginchern}
We can express the Pontryagin classes of any complex vector bundle $E\to X$ in terms of its Chern classes, via
\begin{equation}\label{eq:Pontryagin}
	p_k(E)=c_k(E)^2+2\sum_{i=1}^k (-1)^i c_{k-i}(E)c_{k+i}(E).
\end{equation}
\end{lemma}
This lemma can be found in \cite[Corollary~15.5]{Milnor}. Also the following formula to compute the Chern classes of the tensor product of a vector bundle and a line bundle is well-known, see e.g.\ \cite[Exercise 4.4.6]{Huybrechts}.
\begin{lemma}\label{lem:cherntensorproduct}
	Let $E\to X$ be a complex vector bundle of rank $2n$ and $L\to X$ a complex line bundle. Then
\begin{equation}\label{eq:chern}
	c_k(E\otimes_\CC L) = \sum_{j=0}^k {2n-j \choose k-j}c_j(E) c_1(L)^{k-j}.
\end{equation}
\end{lemma}

\section{Bundle decompositions and characteristic classes}

In this section we start the proof of Theorem \ref{thm:main}. We use the notation of Section \ref{subsec:aqs}, i.e., $M$ is an almost quaternionic manifold of dimension $4n$, endowed with the almost quaternionic structure $Q\subset \End(TM)$. Let $Z$ be the twistor space, as an almost complex manifold as described in Section \ref{subsec:aqs}, and recall the decomposition $TZ=V\oplus H$ of complex vector bundles.

\begin{proposition}\label{prop:VHquerH}
	We have an isomorphism 
	\[
	V\otimes_\CC \overline{H}\cong H
	\]
	of complex vector bundles. 
\end{proposition}
\begin{proof} 
	Recall that for a point $I\in Z$ the complex structure $J_I^H$ on $H_I$ is defined by pulling back the complex structure $I$ on $T_{\pi(I)}M$ via the isomorphism $d\pi_I:H_I\to T_{\pi(I)}M$. We define an isomorphism $\Psi:V\otimes \overline{H}\to H$ as follows: on the fiber over $I$ we put
	\[
	\Psi(J\otimes v)= (d\pi_I)^{-1}(J d\pi_I(v)).
	\]
	This is a well-defined complex linear map on the complex tensor product $V\otimes \overline{H}$, as the corresponding map on the Cartesian product $V\times \overline{H}\to H$ is $\CC$-bilinear:
	\[
	\Psi((I J), v)= (d\pi_I)^{-1}(IJ d\pi_I(v)) = J_I^H(\Psi(J,v))
	\]
	and
	\begin{align*}
	\Psi(J,-J_I^H(v)) & = -(d\pi_I)^{-1}(Jd\pi_I(J_I^H(v))) \\
	&=-(d\pi_I)^{-1}(JId\pi_I(v)) \\
	&= (d\pi_I)^{-1}(IJd\pi_I(v)) = J_I^H(\Psi(J,v)).
	\end{align*}
\end{proof}
One idea that enters the proof of Theorem \ref{thm:main} is to recursively determine the Chern classes of $H$. More precisely, we use Lemma \ref{lem:pontryaginchern} to write an even Chern class of $H$ in terms of lower Chern classes and the Pontryagin classes of $H$ (i.e., the Pontryagin classes of $M$, because of Lemma \ref{lem:identitiesH}). To determine the odd Chern classes of $H$ we apply Lemma \ref{lem:cherntensorproduct} to the isomorphism from Proposition \ref{prop:VHquerH} and obtain
\begin{equation} \label{eq:oddck}
	c_k(H) = \sum_{j=0}^k (-1)^j{2n-j \choose k-j} c_j(H)c_1(V)^{k-j}.
\end{equation}
For odd $k$, this formula determines $c_k(H)$ in terms of the lower Chern classes of $H$ and $c_1(V)$. Let us already note the following immediate consequence:
\begin{proposition}\label{prop:c1}
	We have $c_1(H)=n c_1(V)$ and $c_1(Z) = (n+1)c_1(V)$. In particular, the first Chern class of $Z$ is divisible by $n+1$.
\end{proposition}
\begin{proof}
	Equation \eqref{eq:oddck}, for $k=1$, reads 
	\[
	c_1(H) = 2n c_1(V) - c_1(H),
	\] 
	i.e., $c_1(H) = nc_1(V)$. The second equation then follows immediately from $TZ = V\oplus H$.
\end{proof}
In the context of complex contact manifolds (hence in particular for the twistor space of quaternion-Kähler manifolds), this statement was known before; it was shown by Kobayashi \cite{kobayashi} by determining explicitly the transition functions of the canonical bundle.

In \cite[Proposizione 12.2]{marchiafava-romani} Marchiafava--Romani show that every almost quaternionic manifold
of dimension $8n$ is spin. The authors in fact consider arbitrary almost quaternionic vector bundles, and make use of a splitting principle for those. We would like to give another proof of the above fact.

\begin{corollary}\label{Cor: AQ in dim 8n are spin}
Suppose $M$ is of dimension $8k$ and admits an almost quaternionic structure. Then $M$ is
spin.
\end{corollary}

\begin{proof}
We recall that $\pi^{\ast}(TM) \cong H$, thus $\pi^{\ast}(w_{2}(M))  = w_{2}(H)$. Now 
by Proposition \ref{prop:c1} $c_{1}(H) = 2k c_{1}(V)$ and since $w_{2}$ is the mod $2$
reduction of $c_{1}$, we infer $w_{2}(H)=0$. Moreover, using the Gysin sequence, we note
that $\pi^{\ast}$ is injective for every $2$-sphere bundle on cohomology in degree $2$.
Thus $w_{2}(M)=0$ and since $M$ is orientable, it follows that $M$ is spin.
\end{proof}

In Section \ref{sec:cpn} we will need a further bundle decomposition, this time of the complexified tangent bundle of the $4n$-dimensional almost quaternionic manifold $M$. We will assume only that $Q$ admits a ${\mathrm{Spin}}^c$ structure, which is equivalent to the third integral Stiefel-Whitney class $W_3(Q)\in H^3(M)$ being zero.

We have ${\mathrm{Spin}}^c(3) = ({\mathrm{Spin}}(3)\times \Un(1))/\ZZ_2 \cong \Un(2)$; under this identification, the canonical projection ${\mathrm{Spin}}^c(3)\cong \Un(2) \to \SO(3)$ is given by the adjoint representation $\Un(2)\to \SO({\mathfrak{su}}(2))\cong \SO(3)$ of $\Un(2)$ on $\mathfrak{su}(2)$. 
This shows the following well-known (cf.\ \cite[Section 5.6]{Donaldson}) reformulation of the existence condition of a ${\mathrm{Spin}}^c$ structure on $Q$:

\begin{proposition}\label{prop:spincstructureofQ}
	Assume that the third integral Stiefel-Whitney class $W_3(Q)$ vanishes. Then there is a
	Hermitian vector bundle $F \to M$ of rank $2$ such that
	\[
	\mathfrak{su}(F) \cong Q
	\]
where
\[
	\mathfrak{su}(F) = \{A \in \mathrm{End}(F) \mid A^{\ast} = -A,\, \mathrm{tr}(A)=0\}.
\]

\end{proposition}

\begin{remark}
Proposition \ref{prop:spincstructureofQ} implies that the twistor fibration $Z \to M$ can be described as a
projectivization of a complex vector bundle of rank $2$, whenever $W_3(Q)=0$. Indeed, we have that $Z$ is just the projectivization $\PP(F)$ of $F$: denoting by $P$ the $\Un(2)$-principal bundle of Hermitian frames in $F$, we have $F = P\times_{\Un(2)} \CC^2$, and hence $\PP(F) = P\times_{\Un(2)} \CC \PP^1$. Now, $\CC \PP^1$ is $\Un(2)$-equivariantly diffeomorphic to the unit sphere $S({\mathfrak{su}}(2))$ in ${\mathfrak{su}}(2)$; this shows that $\PP(F) \cong P\times_{\Un(2)} S({\mathfrak{su}}(2))$, which is the unit sphere bundle in $P\times_{\Un(2)} {\mathfrak{su}}(2)\cong{\mathfrak{su}}(F) \cong Q$.
\end{remark}
\begin{remark}\label{Rem: freedom of choosing F}
Note that the bundle $F \to M$ of Proposition \ref{prop:spincstructureofQ} is not unique.
In fact, if $L \to M$ is a Hermitian line bundle, then
\[
  \mathfrak{su}(F) \longrightarrow \mathfrak{su}(F \otimes_\CC L), \quad A \mapsto A \otimes
	\id_{L}
\]
is an isomorphism of real vector bundles. For all choices of $F$ we have the equality of Stiefel-Whitney classes
\begin{equation}\label{eq:w2ofF}
w_2(F) = w_2(Q),
\end{equation}
see \cite[Section 5.6]{Donaldson}.
\end{remark}

\begin{proposition}\label{P: TM_C = E ox F}
	For any $F\to M$ such that ${\mathfrak{su}}(F)\cong Q$, there exists a complex rank $2n$ vector bundle $E \to  M$ such that
	\[
	TM \otimes \mathbb{C} \cong E \otimes_\CC F. 
	\]
\end{proposition}

\begin{proof}
	First note that $Q \otimes \mathbb{C} \subset \mathrm{End}(W)$ where $W = TM \otimes \mathbb{C}$. We claim that $\mathrm{End}(F)$ embeds into
	$\mathrm{End}(W)$ and turns $W$  into a $\mathrm{End}(F)$ module. Note that
	\[
	\mathrm{End}(F) = \mathbb{C} \cdot \id_{F} \oplus \, \mathfrak{sl}(F) =
	\underline{\mathbb{C}} \oplus \mathfrak{sl}(F)
	\]
	where 
	\[
	\mathfrak{sl}(F) = \{ A \in \mathrm{End}(F) \mid \mathrm{tr} A = 0\}.
	\]
	Since by Proposition \ref{P: Chern classes of F} $Q \otimes \mathbb{C} \cong \mathfrak{su}(F) \otimes
	\mathbb{C}  = \mathfrak{sl}(F)$ we define a map 
	\[
	\mathrm{End}(F) \cong \underline{\mathbb{C}} \oplus (Q \otimes \mathbb{C}) \to \mathrm{End}(W), \quad  (\lambda, A) \mapsto \lambda \cdot \id_{W} + A.
	\]
	We show that this map is an embedding, which is equivalent to argue that $\mathbb{C} \cdot
	\id_{W} \cap (Q \otimes \mathbb{C}) = \{0\}$. But since $Q$ is locally spanned by $I,J,K$
	and the addition of $\id_{W}$ gives an action of the complexified quaternion algebra
	acting on $W$ the above intersection must be trivial.
	
	Now set
	\[
	E:= \operatorname{Hom}_{\mathrm{End}(F)}(F, W) 
	\]
	i.e. all vector bundle morphisms $F \to W$ which are equivariant with respect to the $\mathrm{End}(F)$-algebra representations. We claim that the map
	\[
	E \otimes_\CC F \longrightarrow W, \quad (A \otimes v) \mapsto A(v) 
	\]
	is an isomorphism. To see this, consider the situation pointwise. By \cite[Section XVII.4, in particular Corollary 4.6]{Lang}, there is up to isomorphism only one simple ${\End(\CC^2)}$-module, so that the ${\End (F_p)}\cong {\End(\CC^2)}$-module $W_p$ is isomorphic to a direct sum of copies of $\CC^2$. This implies that $W_p \cong \CC^2\otimes_\CC \CC^{2n}$ (with the action only on the left factor), and this isomorphism is just the isomorphism in question, with $E_p \cong \CC^{2n}$.
\end{proof}

\begin{remark}
	It is instructive to compare this to Salamon's $E$-$H$-formalism \cite{Salamon}. An almost quaternionic structure (together with the choice of a compatible metric) gives rise to a reduction of the structure group to $\Sp(n)\Sp(1)$. In this formalism, one considers locally defined complex vector bundles $E$ and $H$ of rank $2n$ and $2$, respectively, associated to the standard representations of $\Sp(n)$ and $\Sp(1)$ on $\CC^{2n}$ and $\CC^2$. (Note that we used $H$ already for the horizontal bundle over the twistor space; we will not use the local bundles $E$ and $H$ on $M$ outside of this remark.) In this context, the isomorphism $TM\otimes \CC = E\otimes_\CC H$ is well-known (in particular, the right hand side is globally defined, although $E$ and $H$ are not); see, for instance, \cite[Section 2]{Cabrera}. Note that the bundle $H$ naturally has structure group $\SU(2)$, and is defined only locally. To compare, our bundle $F$ has structure group $\Un(2)$, and is defined whenever $W_3(Q)$ vanishes. In the situation that is of interest to us, see Proposition \ref{P: Chern classes of F} below, $w_2(Q)\neq 0$ so $Q$ does not have a spin structure.
\end{remark}

To proceed with the proof of Theorem \ref{thm:main} we now consider the different compact rank one symmetric spaces separately. Real projective spaces $\RR P^{4n}$ cannot admit any almost quaternionic structure as they are not orientable.

\section{Spheres}\label{sec:spheres}

It is known that spheres other than $S^4$ do not admit an almost quaternionic structure. In the mathoverflow thread \cite{overflowmoroianu} this was justified in several ways; one of them was to apply \cite[Theorem 1.4]{Sutherland} to see that the total space of an $S^2$-bundle over $S^{4n}$, $n\geq 2$, never admits an almost complex structure, making the existence of a twistor space of $S^{4n}$ impossible. Let us give an alternative short argument using our techniques.

Assume that $S^{4n}$, $n\geq 2$, had an almost quaternionic structure, and denote by $Z$ its twistor space. As the tangent bundle of $S^{4n}$ is stably trivial, all Pontryagin classes of $S^{4n}$ are zero, hence because of $H=\pi^*TS^{4n}$ (see Lemma \ref{lem:identitiesH}) we have $p_k(H)=0$ for all $k\geq 1$. Applying Lemma \ref{lem:pontryaginchern} to $H$, we thus obtain
\begin{equation}\label{eq:Pontryagin0}
c_{n}(H)^2 + 2\sum_{i=1}^n (-1)^i c_{n-i}(H)c_{n+i}(H) = p_n(H) = 0.
\end{equation}
As $S^{4n}$ has nonzero cohomology only in degrees $0$ and $4n$, the cohomology of $Z$
vanishes in all degrees except $0,2,4n$ and $4n+2$. As $n\geq 2$, it follows that the only
term in the left hand side of Equation \eqref{eq:Pontryagin0} that does not automatically
vanish is the summand for $i=n$, i.e., it follows that $(-1)^n 2 c_{2n}(H)=0$, i.e.,
$c_{2n}(H)=0$. This is a contradiction because the top Chern class equals the Euler class,
$H^{4n}(S^{4n})\to H^{4n}(Z)$ is an isomorphism by the Gysin sequence and the Euler characteristic $\chi(S^{4n})=2$ of $S^{4n}$ is nonzero.

\section{The Cayley plane}\label{sec:cayley}

In this section, let us consider the Cayley plane $\mathbb{O}P^2$, and show that it does not admit an almost quaternionic structure. For the sake of contradiction, let us assume it had, and denote by $Z$ the corresponding twistor space. 

Let $u\in H^8({\mathbb{O}}P^2)$ be a generator, so that the cohomology ring $H^*({\mathbb{O}}P^2)= \ZZ[u]/(u^3)$. After potentially replacing $u$ by $-u$, the nontrivial Pontryagin classes of ${\mathbb{O}}P^2$ are
$p_2(\mathbb{O}P^2)=6u$ and $p_4(\mathbb{O}P^2)=39u^2$, see \cite[\S 19]{BorelHirzebruch}.  The fiber inclusion $S^2\to Z$ induces an isomorphism $\ZZ = \pi_2(S^2)\to \pi_2(Z)$, hence by Hurewicz also an isomorphism $H_2(S^2)\to H_2(Z)$. By the universal coefficient theorem, we obtain an isomorphism $H^2(Z)\to H^2(S^2)$, so that the assumptions of the Leray-Hirsch theorem over the integers are satisfied; as necessarily $x^2=0$, the cohomology ring of $Z$ can be written as
\[
H^*(Z)\cong \ZZ[u,x]/(u^3,x^2)
\]
where $x\in H^2(Z)$ restricts to a generator of the fiber $S^2$. In particular, $Z$ has nonzero cohomology in degrees $0,2,8,10,16$ and $18$, so that
automatically $c_2(H)=c_3(H) = c_6(H)=c_7(H)=0$.

As $c_3(H)=0$, Lemma \ref{lem:pontryaginchern} yields $6u = p_2(H) =  2c_4(H)$,  i.e.
\[
c_4(H) = 3u,
\]
and then by the same lemma $39u^2 = p_4(H) = c_4(H)^2 +2c_8(H)$ (again, all other summands vanish), i.e.
\[
c_8(H) = 15u^2.
\]
This is a contradiction, as $\chi({\mathbb{O}}P^2) = 3$ and $\pi^*:H^{16}({\mathbb{O}}P^2)\to H^{16}(Z)$ is an isomorphism by the Gysin sequence.

\section{Complex projective spaces}\label{sec:cpn}
Let us come now to the most interesting case, that of $M:=\CC \PP^{2n}$. Note that it is already known that $\CC \PP^{4n}$ is not almost quaternionic, as almost quaternionic manifolds of dimension $8n$ are spin \cite{marchiafava-romani}, a fact reproven above in Corollary \ref{Cor: AQ in dim 8n are spin}. As explained in the introduction, we will use K-theory in this section; see Section \ref{rem:remark} below for a different proof of a part of our main theorem using the Chern class approach from the previous sections.

The cohomology ring of $\mathbb C\mathbb P^{N}$, where $N=4n+2$, is the truncated polynomial
ring
\[
	H^{\ast}(\mathbb C\mathbb P^{N}) = \mathbb{Z}[u] / (u^{N+1}).
\]
We fix as generator $u \in H^{2}(\mathbb C\mathbb P^{N};\mathbb{Z})$ the first Chern class $c_1(H)$ of the dual $H\to \CC \PP^n$ of the tautological bundle (do not confuse this with the horizontal bundle on the twistor space, which does not appear in this section). Let us refine the statement of Proposition \ref{prop:spincstructureofQ} in this setting.

\begin{proposition}\label{P: Chern classes of F}
Assume that $\mathbb C\mathbb P^{N}$ admits an almost quaternionic structure.
Then there is a complex rank $2$ vector bundle $F \to \mathbb C\mathbb P^{N}$ such that
\[
	\mathfrak{su}(F) \cong Q, \quad c_{1}(F) = u, \quad c_{2}(F) = 2m u^{2} 
\]
for some $m \in \mathbb{Z}$.
\end{proposition}

\begin{proof}
Let $F\to \CC \PP^N$ be a bundle as in Proposition \ref{prop:spincstructureofQ}, i.e., such that ${\mathfrak{su}}(F) \cong Q$. Recall from Remark \ref{Rem: freedom of choosing F} that  we may tensor $F$ with any line bundle $L \to \mathbb C\mathbb P^{N}$, without destroying the isomorphism $\mathfrak{su}(F \otimes_\CC L) \cong Q$. 

By Equation \eqref{eq:w2ofF} we have the equality of second Stiefel-Whitney classes $w_{2}(F) = w_{2}(Q)\in H^2(\CC \PP^N;\ZZ_2)$, and from \cite[Proposizione 12.2]{marchiafava-romani} we obtain 
\[
w_2(Q) = (2n+1)w_{2}(Q) = w_{2}(\mathbb C\mathbb P^{N})
= u \mod 2
\] since $\dim \mathbb C\mathbb P^{N} = 8n+4$. This implies that $c_{1}(F) = (2l +1)\cdot u$ for some $l\in \ZZ$. Let $L
\to \mathbb C\mathbb P^{N}$ be the line bundle defined by $c_{1}(L) = -l \cdot u$, then by
formula \eqref{eq:chern}
\[
	c_{1}(F \otimes_\CC L)  = c_{1}(F) + 2c_{1}(L) = u.
\]
So we may assume that $c_{1}(F)=u$. It remains to show that $c_{2}(F)$ is an even
multiple of $u^{2}$. As $w_{2k}(F) \equiv c_{k}(F) \mod 2$, we only need to show that $w_4(F)=0$. Let $c_{2}(F) = a \cdot u^{2}$ for $a \in \mathbb{Z}$. By the Wu
formula (see \cite[p. 197]{MR1702278})  we have 
\begin{align}\label{eq:Wu Formula}
	\mathrm{Sq}^{2}(w_{4}(F)) = w_{2}(F)w_{4}(F) + w_{6}(F) = w_{2}(F)w_{4}(F)
\end{align}
since $w_{6}(F)=0$ because the underlying real vector bundle of $F$ has rank $4$.  If $\overline{u}$ denotes the mod $2$ reduction
of $u$, then 
\[
  \mathrm{Sq}^{2}(\overline{u}^{2}) = {2 \choose 1}\overline{u}^{3} = 0  
\]
and therefore by Equation \eqref{eq:Wu Formula}
\[
	0 = \mathrm{Sq}^{2}(a \overline{u}^{2}) = a \cdot \overline{u}^{3}
\]
from which we infer that $a$ must be even.
\end{proof}
Let us apply Proposition \ref{P: TM_C = E ox F} to the bundle $F\to \CC \PP^N$ from Proposition \ref{P: Chern classes of F}, in order to obtain a complex rank $N$ vector bundle $E\to \CC \PP^N$ such that
\begin{equation}\label{eq:tensorproddecomp}
W:=T(\CC \PP^N)\otimes \CC \cong E\otimes_\CC F.
\end{equation}

Next, we recall some basic facts on the K-theory of complex projective spaces. Denote by
$K:=K^0(\mathbb C\mathbb P^{N})$ and put $z = H-1 \in K$. Then we have
\[
	K \cong \mathbb{Z}[z]/(z^{N+1}).
\]
For $r \in \mathbb{N}$ let $K_{r} := K/(2^{r} \cdot K)$. We obtain an evaluation map
\[
	\varepsilon \colon K_{N+1} \longrightarrow \mathbb{Z}_{2^{N+1}}, \quad z \mapsto -2 \mod
	2^{N+1}
\]
which is a well-defined ring homomorphism since $\varepsilon(z^{N+1}) = (-2)^{N+1} \equiv
0 \mod 2^{N+1}$. We will create a contradiction to \eqref{eq:tensorproddecomp} by showing $\varepsilon(W)\neq 0$ but
$\varepsilon(F)=0$. 

\begin{proposition}\label{P: eps of W is not zero}
We have $\varepsilon(W) \neq 0$.
\end{proposition}

\begin{proof}
Denote by $C$ the complex tangent bundle of $\mathbb C\mathbb P^{N}$ as complex manifold.
Thus $W \cong C \oplus \overline{C}$. Furthermore, it is known (e.g., \cite[p.\ 170]{Milnor})
that
\[
	C \oplus \underline{\mathbb{C}} \cong \bigoplus_{i=1}^{N+1} H
\]
hence in $K$ we have $C = (N+1)H -1$ and therefore
\[
	W = C + \overline{C} = (N+1)(H + \overline{H}) -2.
\]
Denote by $\overline{z} = \overline{H}-1$. Then, since $H \cdot \overline{H} = 1$ we have
\[
	1+\overline{z} = \frac{1}{1+z}= \sum_{k=0}^{\infty}(-z)^{k} = \sum_{k=0}^{N}(-z)^{k}
\]
whence
\[
	\overline{z} = \sum_{k=1}^{N}(-z)^{k}.
\]
We compute
\begin{align*}
	W &= (N+1)(H + \overline{H}) -2\\
		&= (N+1)\left(2+\sum_{k=2}^{N}(-z)^{k}\right) -2\\
\end{align*}
and therefore
\[
	\varepsilon(W) \equiv (4n+3)(2^{4n+3}-2)-2 \equiv -8(n+1) \mod 2^{4n+3}
\]
which is never zero, since $n+1 < 2^{4n}$ for $n \geq 1$.
\end{proof}

Now, we turn our attention to the computation of $\varepsilon(F)$. There is no reason that
$F$ splits into a sum of line bundles, but we will show this happens in $K_{2^{N+1}}$.
From there, it will be possible to compute $\varepsilon(F)$.

If $F$ was a sum of line bundles $H^{a} \oplus H^{b}$,
then the sum would have the same Chern classes as $F$, i.e., from Proposition \ref{P:
Chern classes of F} we would obtain
\[
	a +b =1, \quad ab = 2m. 
\]
Thus, these numbers would be the solutions to the quadratic equation
\[
	X^{2} - X + 2m =0 
\]
and therefore
\[
	a,b = \frac{1 \pm \sqrt{1 - 8m}}{2}.
\]
These numbers are in general not integers. In fact, we take this computation only as
motivation for the following considerations modulo some power of $2$. It is well-known
that all positive numbers congruent to $1 \mod 8$ are quadratic residues modulo any power
of $2$, see for instance \cite[Theorem 2]{Dence}. Thus, for every
positive $r \in \mathbb{N}$ there exists $s_{r} \in
\mathbb{Z}$ such that
\[
	s_{r}^{2}  \equiv 1 - 8m \mod 2^{r}.
\]
We set 
\[
	a_{r} =  \frac{1-s_{r}}{2},  \quad b_{r} = \frac{1+s_{r}}{2}.
\]
Then
\begin{equation}\label{eq: approximated chern roots}
a_{r} + b_{r} =1, \quad a_{r}b_{r} = \frac{1-s_{r}^{2}}{4} \equiv 2m \mod 2^{r}.
\end{equation}

\begin{proposition}\label{P: F in K_r}
For $r$ big enough we have 
\[
	F \equiv H^{a_{r}} + H^{b_{r}} \mod 2^{N+1}K. 
\]
\begin{proof}
Set $G_{r} = H^{a_{r}} \oplus H^{b_{r}}$. We start by noticing that there are unique $d_{r,j} \in \mathbb{Z}$ such that
\[
	F - G_{r}  = \sum_{j=0}^{N} d_{r,j} z^{j}\in K.
\]

Next, let $\mathrm{ch} \colon K(\mathbb C\mathbb P^{N}) \to H^{\ast}(\mathbb C\mathbb P^{N};\mathbb{Q})$ be the Chern character of $\mathbb C\mathbb P^{N}$. Since the cohomology of $\mathbb C\mathbb P^{N}$ is
torsion-free, we know that $\mathrm{ch}$ is injective, see \cite[Theorem on p.\
131]{Hirzebruch}. Furthermore we have

\[
	\mathrm{ch}(F) -\mathrm{ch}(G_{r})  = \sum_{j=0}^{N} d_{r,j} x^{j}
\]
for $x = \mathrm{ch}(z) = \mathrm{ch}(H-1) = e^{u}-1$. Note that by the injectivity the
image of $\mathrm{ch}$ is a lattice with basis $\{1, x, \ldots, x^{N}\}$.
Recall that for the complex bundle $F \to X$ of rank $2$ we may write
\[
	\mathrm{ch}(F) = \sum_{j=0}^{\infty}\frac{1}{j!}(\nu_{j}(c_{1}(F), c_{2}(F)))
\]
where the $\nu_{j}$ are the power-sum symmetric polynomials in the respective Chern roots, which by the Newton identities can be rewritten as polynomials in the respective Chern classes. For the difference, we obtain
\[
	\mathrm{ch}(F) - \mathrm{ch}(G_{r}) = \sum_{j=0}^{\infty} 
	\frac{1}{j!}(\nu_{j}(c_{1}(F),c_{2}(F)) - \nu_{j}(c_{1}(G_{r}), c_{2}(G_{r})))
\]
From Equation \eqref{eq: approximated chern roots} we infer $c_{1}(F) = c_{1}(G_{r})$ and
$c_{2}(F) \equiv c_{2}(G_{r}) \mod 2^{r}$. Thus, there is a polynomial $q_{j}$ with
integer coefficients such that
\[
	\mathrm{ch}(F) - \mathrm{ch}(G_{r}) = \sum_{j=0}^{\infty} \frac{1}{j!}(q_{j}(f) -
	q_{j}(g_{r}))u^{j}
\]
where $c_{2}(F) = f \cdot u^{2}$ (with $f=2m$) and $c_{2}(G_{r}) = g_{r} \cdot u^{2}$. Write
\[
	q_{j}(f) - q_{j}(g_{r}) = (f-g_{r}) \cdot h_{j}(f,g_{r}) 
\]
for some polynomial $h_j$ with integer coefficients.
Moreover, we have 
\[
	u = \log(1+x) = \sum_{k=1}^{N}\frac{(-1)^{k-1}}{k} x^{k}
\]
and therefore 
\[
	\mathrm{ch}(F) - \mathrm{ch}(G_{r})  = \sum_{j=0}^{N} (f-g_{r})
	\frac{k_{r,j}}{l_{j}}x^{j}
\]
for some $k_{r, j},l_{j} \in \mathbb{Z}$. Hence $(f-g_{r}) \frac{k_{r,j}}{l_{j}} = d_{r,j}$.
Since the sum is finite there is an $L \in \mathbb{Z}$ independent of $r$ such that $l_{j}
| L$ for all $j$. Hence,
\[
	(f-g_{r})\, |\, d_{r,j} L.
\]
The left hand side is divisible by $2^{r}$, so we may choose some large $r$ such that $2^{N+1}$
divides $d_{r,j}$ for all $j$. Thus
\[
	F - G_{r} \in 2^{N+1} K.
\]
\end{proof}
\end{proposition}

Finally we show

\begin{proposition}
We have $\varepsilon(F)=0$.
\end{proposition}

\begin{proof}
From Proposition \ref{P: F in K_r}  we have in $K_{N+1}$
\[
	\varepsilon(F) = \varepsilon(H^{a_{r}}) + \varepsilon(H^{b_{r}}) = (-1)^{a_{r}} +
	(-1)^{b_{r}}
\]
since $z = H-1$. The relation $a_{r}+b_{r} = 1$ shows that $a_{r}$ and $b_{r}$ must have
different parity mod $2$ and therefore the right-hand side is equal to zero.
\end{proof}

As discussed above, since $\varepsilon$ is a ring homomorphism, we obtain a contradiction with Proposition
\ref{P: eps of W is not zero} via
\[
	0 \neq \varepsilon(W)  = \varepsilon(E)\varepsilon(F) = 0.
\]
Thus the assumption that $\mathbb C\mathbb P^{N}$ possesses an almost quaternionic structure must have been false.

\section{A remark: Complex projective spaces via Chern classes}\label{rem:remark}

Although we already proved our main theorem, let us revisit the case of $M:=\CC \PP^{2m}$ using our approach via Chern classes. We will show that in all cases except when $2m$ is of the form $2m=2^n-2$, our recursion formulas for the Chern classes of the horizontal bundle $H$ over the twistor space $Z$ give a contradiction to their integrality.

The cohomology ring of $\mathbb C\mathbb P^{2m}$ is given by
$H^{\ast}(\mathbb C\mathbb P^{2m}) = \mathbb{Z}[u]/(u^{2m+1})$ for a generator $u \in H^{2}(\mathbb
C\mathbb P^{2m}) \cong \mathbb{Z}$. As in the case of the spheres and the Cayley plane, the assumptions for the Leray-Hirsch theorem are satisfied, in order to compute the cohomology of $Z$.

Using Lemma \ref{lem:QV} we have
\begin{equation}\label{eq:p1QV}
	\pi^* p_1(Q) = p_1(\pi^*Q) = p_1(V) = c_1^2(V).
\end{equation}
 The Pontryagin classes of complex projective space are given by
\begin{equation}\label{eq:pontryaginCPn}
	p_s(M) = {2m+1 \choose s}u^{2s}
\end{equation}
(see \cite[Example 15.6]{Milnor}), so in particular $p_1(M) = (2m+1)u^2$, hence from Lemmas \ref{lem:identitiesH} and \ref{lem:pontryaginchern} as well as Proposition \ref{prop:c1} we get
\begin{equation}
	(2m+1)\pi^*u^2 =  p_1(H) = c_1(H)^2-2c_2(H) = m^2c_1(V)^2 - 2c_2(H).
\end{equation}
Hence, if $m$ is even, then we obtain a contradiction, as $(2m+1)\pi^*u^2$ is not divisible by $2$ in $H^*(Z)$. (Of course we knew already that $m$ cannot be even by Corollary \ref{Cor: AQ in dim 8n are spin}, respectively by \cite{marchiafava-romani}.)

The same argument implies that $c_1(V)^2$, and hence also $c_1(V)$, is not divisible by $2$ in $H^2(Z)$. Moreover $V$, when restricted to a fiber, is the tangent bundle of $S^2$; it follows that we may choose $x\in H^2(Z)$ that restricts to a generator of $H^2(S^2;\ZZ)$ such that
\[
c_1(V) = 2x+\pi^*u.
\] We compute $4x^2 + 4x\pi^*u  + \pi^*u^2= c_1(V)^2 = \pi^*p_1(Q)$ using \eqref{eq:p1QV}. It follows that there is $k\in \ZZ$ such that 
\begin{equation}\label{eq:p1ink}
	p_1(Q) = (4k+1)u^2
\end{equation}
and
\begin{equation}\label{eq:cohomringCPn}
	H^*(Z) \cong \ZZ[u,x]/(u^{2n+1},x^2+xu-ku^2).
\end{equation}

Let us now assume that the even number $2m$ is not of the form $2^n-2$ for any $n$. Consider the binary expansion of $2m$: its least significant bit is $0$, and the remainder of the expansion does not consist entirely of $1$'s. By looking at the next appearing $0$, we see that we find natural numbers $l\geq 2$ and $r\geq 1$ such that $2m = 2^lr+2^{l-1}-2$ (the summand $2^{l-1}-2$ coming from the $1$'s following the least significant bit $0$).

Consider the $(2^{l-1}-1)$st Chern class of $M$. Lemma \ref{lem:pontryaginchern} gives
\begin{equation}\label{eq:pontrchern1}
	p_{2^{l-1}-1}(H) \equiv c_{2^{l-1}-1}(H)^2 \quad \mod 2.
\end{equation}
By \eqref{eq:pontryaginCPn} we have 
\[
p_{2^{l-1}-1}(H) = {2^lr+2^{l-1}-1 \choose 2^{l-1}-1} \pi^*u^{2^l-2}.
\]
Recall that it follows from Kummer's theorem \cite{Kummer} that the maximal exponent of $2$ dividing a binomial coefficient ${a \choose b}$ equals the number of carries that occur when adding $b$ and $a-b$ in base $2$. For the binomial coefficient above, adding $2^{l-1}-1$ to $2^lr$ does not produce any carry, so that $p_{2^{l-1}-1}(H)$ is not divisible by $2$. On the other hand,  Equation \eqref{eq:oddck} gives
\[
2c_{2^{l-1}-1}(H) = \sum_{j=0}^{2^{l-1}-2} (-1)^j {2^lr+2^{l-1}-2-j \choose 2^{l-1}-1-j} c_j(H)c_1(V)^{2^{l-1}-1-j}
\]
This time, for any $j$, adding $2^{l-1}-1-j$ to $2^l r-1$ in base $2$ produces at least two carries (as $l\geq 2$), so that all binomial coefficients in the above sum are divisible by $4$. It follows that $c_{2^{l-1}-1}(H)$ is divisible by $2$. As we saw above that $p_{2^{l-1}-1}(H)$ is not divisible by $2$, this contradicts Equation \eqref{eq:pontrchern1}.

Using this approach, the cases  $\CC \PP^{2^n-2}$ are not covered. 	We believe it should be possible to exclude the existence of almost quaternionic structures on $\CC \PP^{2^n-2}$ using these techniques, via the equality between the top Chern class and the Euler class as for spheres and the Cayley plane above. We were, however, unable to determine the relevant top Chern classes in a more tractable way than the recursive formulas in order to attack this question in full generality; only for $n=3$ and $n=4$ we can solve it. To this end, we first compute recursively the Chern classes of $H$, as a polynomial in the variable $k$ (see Equation \eqref{eq:p1ink}) using Lemma \ref{lem:pontryaginchern} for the even ones and Equation \eqref{eq:oddck} for the odd ones, and the description of the cohomology ring of $Z$ in \eqref{eq:cohomringCPn}. For $\CC \PP^{2^3-2}=\CC \PP^6$ we obtain
\begin{align*}
%	c_1(H) &= 3c_1(V)\\
%	c_2(H) &= (18k+1)u^2\\
%	c_3(H) &= (16k-3)u^2c_1(V)\\
%	c_4(H) &= (30k^2-6k+1)u^4\\
%	c_5(H) &= (6h^2-4k+3)u^4c_1(V)\\
	c_6(H) &= (44k^3-16k^2-18k-5) u^6.
\end{align*}
This is a contradiction because there is no integer $k$ that solves the equation
\[
44k^3-16k^2-18k-5 = \chi(\CC \PP^6) = 7.
\]
For $\CC \PP^{2^4-2} = \CC \PP^{14}$ we obtain (with the help of a Python script):
\begin{align*}
%	c_1(H) &= 7c_1(V)\\
%	c_2(H) &=( 98k + 17) u^2\\
%	c_3(H) &= (224k + 11)u^2 c_1(V)\\
%	c_4(H) &= (1470k^2 + 210k - 15) u^4\\
%	c_5(H) &= (1806k^2 - 72 k - 9) u^4 c_1(V)\\
%	c_6(H) &= (6860 k^3 - 150k + 25) u^6\\
%	c_7(H) &= (5216k^3 - 1188k^2 + 48k - 11) u^6c_1(V)\\
%	c_8(H) &= (11494k^4 - 2576 k^3 + 234k^2 + 74k - 31) u^8\\
%	c_9(H) &= (2002k^4 - 2648k^3 + 2202k^2 + 492k + 87) u^8 c_1(V)\\
%	c_{10}(H) &= (42252k^5 + 450k^4 - 15300k^3 - 8420k^2 - 1450k - 71)u^{10}\\
%	c_{11}(H) &= (97440k^5 + 14310k^4 - 43560 k^3 - 24700 k^2 \\ & \qquad\qquad\qquad - 4820 k - 373)u^{10}c_1(V)\\
%	c_{12}(H) &= (-2364180k^6 - 806940k^5 + 1036020k^4 \\ & \qquad \qquad \qquad + 797280k^3 + 222510k^2 + 28278k + 1377) u^{12}\\
%	c_{13}(H) &= (-2567460k^6 - 891360k^5 + 1120260k^4 \\ & \qquad \qquad \qquad  + 873160k^3 + 246210k^2 + 31764k + 1591)u^{12}c_1(V)\\
	c_{14}(H) &= (175023480k^7 + 96140280k^6 - 64359000k^5 - 75024340k^4  \\ & \qquad \qquad \qquad - 28685020k^3 - 5479680k^2 - 529210k - 20605) u^{14}.
\end{align*}
Again, for no integer $k$ this polynomial takes the value $\chi(\CC \PP^{14}) =15$.


\begin{thebibliography}{99}
\bibitem{Alekseevsky} Dmitri V.\ Alekseevskii, \emph{Riemannian spaces with exceptional holonomy groups}, Funct.\ Anal.\ Appl.\ {\bf 2} (1968), no.\ 2, 97--105.
\bibitem{AMP} Dmitri V.\ Alekseevsky, Stefano Marchiafava and Massimiliano Pontecorvo, \emph{Compatible complex structures on almost quaternionic manifolds}, Trans.\ Amer.\ Math.\ Soc.\ {\bf 351} (1999), no.\ 3, 997--1014.
\bibitem{BorelHirzebruch} Armand Borel and Friedrich Hirzebruch, \emph{Characteristic classes and homogeneous spaces. I}, Amer.\ J.\ Math.\ {\bf 80} (1958), 458--538.
\bibitem{CadekVanzura} Martin Čadek and Jiří Vanžura, \emph{Almost quaternionic structures on eight-manifolds}, Osaka J.\ Math.\ {\bf 35} (1998), no.\ 1, 165--190.
\bibitem{Dence} Joseph B.\ Dence and Thomas P.\ Dence, \emph{Residues. II: Congruences modulo powers of $2$}, Missouri J.\ Math.\ Sci.\ {\bf 8} (1996), No.\ 1, 26--35.
\bibitem{Donaldson} Simon K.\ Donaldson, \emph{Floer homology groups in Yang-Mills theory}, Cambridge Tracts in Mathematics. 147. Cambridge: Cambridge University Press, 2002.
\bibitem{Hirzebruch} Friedrich Hirzebruch, \emph{A Riemann-Roch theorem for differentiable manifolds}, Semin.\ Bourbaki {\bf 11} (1958/59), Exp.\ No.\ 177, 21 p.
\bibitem{Huybrechts} Daniel Huybrechts, \emph{Complex geometry. An introduction}. Universitext, Springer (2005).
\bibitem{kobayashi} Shoshichi Kobayashi, \emph{Remarks on Complex Contact Manifolds}, Proc.\ Amer.\ Math.\ Soc.\ {\bf 10} (1959), No.\ 1, 164--167.
\bibitem{Kummer} Ernst E.\ Kummer, \emph{{\"U}ber die {Erg{\"a}nzungss{\"a}tze} zu den allgemeinen {Reciprocit{\"a}tsgesetzen}}, J.\ Reine Angew.\ Math.\ {\bf 44} (1852), 93--146.
\bibitem{marchiafava-romani}
Stafano Marchiafava and Giuliano Romani, \emph{Classi caratteristiche dei fibrati quaternionali generalizzati}, Atti Accad.\ Naz.\ Lincei, VIII. Ser., Rend., Cl.\ Sci.\ Fis.\ Mat.\ Nat.\ {\bf 56} (1974), 899--906.
\bibitem{Cabrera} Francisco Martín Cabrera, \emph{Almost quaternion-Hermitian manifolds}, Ann.\ Global Anal.\ Geom.\ {\bf 25} (2004), No.\ 3, 277--301.
\bibitem{Lang} Serge Lang, \emph{Algebra. 3rd revised ed.}, Graduate Texts in Mathematics {\bf 211}. New York, NY: Springer, 2002.
\bibitem{LeBrunSalamon} Claude LeBrun and Simon Salamon, \emph{Strong rigidity of positive quaternion-Kähler manifolds}, Invent.\ Math.\ {\bf 118} (1994), No.\ 1, 109--132.
\bibitem{MR1702278} Jon Peter May, \emph{A concise course in algebraic topology}, Chicago Lectures in Mathematics, University of Chicago Press (1999).
\bibitem{Milnor} John W.\ Milnor and James D.\ Stasheff, \emph{Characteristic classes}. Ann. of Math. Stud., No. 76
Princeton University Press, Princeton, NJ; University of Tokyo Press, Tokyo, 1974.
\bibitem{overflowmoroianu} Andrei Moroianu, thread on mathoverflow.net,  \url{https://mathoverflow.net/questions/52396/are-there-topological-obstructions-to-the-existence-of-almost-quaternionic-struc}, accessed on April 20, 2026. 
\bibitem{MoroianuPilcaSemmelmann} Andrei Moroianu, Mihaela Pilca and Uwe Semmelmann, \emph{Homogeneous almost quaternion-Hermitian manifolds}, Math.\ Ann.\ {\bf 357} (2013), No.\ 4, 1205--1216.
\bibitem{Salamon} Simon Salamon, \emph{Quaternionic Kähler manifolds}, Invent.\ Math.\ {\bf 67} (1982), 143--171.
\bibitem{Sutherland} Wilson A.\ Sutherland, \emph{A note on almost complex and weakly complex structures}, J.\ Lond.\ Math.\ Soc.\ {\bf 40} (1965), 705--712.
\end{thebibliography}
\end{document}